\documentclass[a4paper, 12pt]{amsart}
\usepackage[utf8]{inputenc}
\usepackage[english]{babel}
\usepackage[T1]{fontenc}
\usepackage{amsmath, amssymb, amsfonts}
\usepackage{enumerate}
\usepackage{graphicx}
\usepackage[hmargin=2.5cm, vmargin=2.5cm]{geometry}
\usepackage{tikz}
\usepackage{comment}
\usepackage{hyperref}
\usepackage{stmaryrd}
\usepackage{xfrac}

\title{Patterns in trees and quantum automorphism groups}
\author{Lucas Alger}
\email{lucas.alger-gaiqui@math.univ-toulouse.fr}
\address{Laboratoire de Math\'ematiques d'Orsay, Univ. Paris-Sud, CNRS, Universit\'e Paris-Saclay, 91405 Orsay, France}
\author{Julie Capron}
\email{julie.capron.hbll@gmail.com}
\address{Laboratoire de Math\'ematiques d'Orsay, Univ. Paris-Sud, CNRS, Universit\'e Paris-Saclay, 91405 Orsay, France}
\author{Félix de la Salle}
\email{felix.de-la-salle@orange.fr}
\address{Laboratoire de Math\'ematiques d'Orsay, Univ. Paris-Sud, CNRS, Universit\'e Paris-Saclay, 91405 Orsay, France}
\date{}

\theoremstyle{plain}
\newtheorem{thm}{Theorem}[section]
\newtheorem{prop}[thm]{Proposition}
\newtheorem{cor}[thm]{Corollary}
\newtheorem{lem}[thm]{Lemma}

\theoremstyle{definition}
\newtheorem{de}[thm]{Definition}

\theoremstyle{remark}
\newtheorem{rem}[thm]{Remark}

\DeclareMathOperator{\QAut}{Aut^{+}}

\newcommand{\E}{\mathbb{E}}

\newcommand{\p}{\mathbb{P}}

\begin{document}

\begin{abstract}
We prove that given a fixed finite tree $P$, almost all trees contain $P$ as a subtree. Moreover, the inclusion can be made so that it induces an embedding of the corresponding (quantum) automorphism groups, thereby providing generic properties of the latter.
\end{abstract}

\maketitle

\section{Introduction}

In the seminal paper \cite{erdos1963asymmetric}, P. Erd\"os and A. R{\'e}nyi proved -- among many other results -- that almost all trees contain what they call a \emph{cherry}, that is to say a line of three vertices attached by its middle to the remainder of the tree. A direct consequence is that the non-trivial automorphism of the cherry consisting in exchanging the endpoints extends to an automorphism of the whole tree, therefore proving that almost all trees have non-trivial automorphism group. This is in sharp contrast with the fact that almost all graphs have trivial automorphism group (and even trivial quantum automorphism group, see \cite[Thm 3.15]{lupini2017nonlocal}).

The notion of quantum graph isomorphism was introduced in \cite{atserias2018quantum} the setting of quantum information theory to understand the difference between classical and quantum strategies for the so-called \emph{graph isomorphism game}. Then, M. Lupini, L. Man\v{c}inska and D. Roberson made in \cite{lupini2017nonlocal} the link with the notion of quantum automorphism group of a graph earlier introduced by J. Bichon and T. Banica (see \cite{bichon2003quantum} and \cite{banica2005quantum}). This opened new perspectives for the study of these objects, and in particular for fundamental questions concerning the mere possibility for a graph of having quantum automorphisms.

The starting point of this work is the article \cite{junk2020almost}, where L. Junk and S. Schmidt and M. Weber improved the result of Erd\"os-R{\'e}nyi to prove that almost all trees contain two disjoint cherries. A consequence of this fact is that, by a criterion of S. Schmidt \cite[Thm 2.2]{schmidt2020quantum}, the corresponding quantum automorphism group is not only non-trivial, but even non-classical. Based on this, it is natural to wonder whether one can similarly prove structure results for ``almost all'' quantum automorphism groups of trees, and this is what we will do.

More precisely, our main result (Theorem \ref{thm:main}) states that given a finite tree $P$ (called the \emph{pattern}), almost all trees contain a copy of $P$ attached through one of its vertices. Up to replacing $P$ by its rootification, this means that one can embed the (quantum) automorphism group of $P$ inside the one of the tree, thereby providing several almost sure properties of quantum automorphism groups of trees.

Let us briefly outline the contents and organization of the paper. Section \ref{sec:preliminaries} contains preliminary material concerning finite graphs and the notion of pattern which is our main tool. Then, the main result is proven in Section \ref{sec:main}. The text ends with the short Section \ref{sec:applications} giving some consequences at the level of quantum automorphism groups.

\section*{Acknowledgments}

This work was done as part of the author's bachelor thesis under the supervision of Amaury Freslon. The authors would like to thank him for countless helpful comments, feedbacks and benevolence for this project.

\section{Preliminaries}\label{sec:preliminaries}

Let us start by recalling a few definitions in order to fix terminology and notations.
\begin{itemize}
\item In this work, the word ``graph'' will always mean ``finite non-oriented graph without loops'', that is to say the data of a set $V$ of vertices and a subset $E\subset V\times V$ not intersecting the diagonal such that $(x, y)\in E$ if and only if $(y, x)\in E$.
\item Two graphs $G$ = $(V, E)$ and $G'$ = $(V', E')$ are said to be \emph{isomorphic} if there exists a bijection $\varphi$ from $V$ to $V'$ such that
\begin{equation*}
\forall\ u,v\in V\ :\ \{u,v\}\in E \Longleftrightarrow \{\varphi(u), \varphi(v)\}\in E'.
\end{equation*}
Such a map $\varphi$ is called a \emph{graph isomorphism}.
\item A graph $G$ is a \emph{tree} if it is connected and does not contain any cycle. If we further fix some specific vertex $v$ of $G$, then $(v, G)$ will be called a \emph{rooted tree}.
\item If $(v,T)$ and $(v',T')$ are two rooted trees, an \emph{isomorphism of rooted trees} is a graph isomorphism that respects the roots i.e. it is a graph isomorphism $\varphi$ from $T$ to $T'$ such that
\begin{equation*}
\varphi (v)=v'.
\end{equation*}
Two rooted trees are said to be \emph{isomorphic} if such a map exists.
\item Let us also recall \emph{Cayley's formula}: The number of possible labelled trees with $n$ vertices is $n^{n-2}$. See \cite{proofsfromthebook} for various proofs.
\end{itemize}

Let us now introduce the main tool of this paper, \emph{patterns}.

\begin{de}
Let $p$ be an integer and let $(v, P)$ be a rooted tree with $p+1$ vertices. Given a tree $T$ with $n\geq p+1$ vertices and $p+1$ distinct vertices
\begin{equation*}
v_{j}, v_{i_{1}}, \cdots, v_{i_{p}}\in V(T),
\end{equation*}
we will say that $(v_{j}, \{v_{j},v_{i_{1}},...,v_{i_{p}}\})$ is a $(v,P)-$\emph{pattern} in $T$ if the subgraph $\widetilde{T}$ of $T$ obtained by restricting to the set $\{v_j,v_{i_1},...,v_{i_p}\}$ satisfies :
\begin{enumerate}
\item There is an isomorphism of rooted trees $\varphi: (v_{j},\widetilde{T})\to (v, P)$;
\item $\mathrm{deg}_{T}(v_{j}) = \mathrm{deg}_{\widetilde T}(v_{j}) + 1$;
\item $\forall\ k \in \llbracket 1, p\rrbracket, \mathrm{deg}_{T}(v_{i_{k}}) = \mathrm{deg}_{\widetilde T}(v_{i_{k}})$.
\end{enumerate}
\end{de}

Let us make a few comments on that definition.

\begin{rem}
\begin{enumerate}
\item We will often write $(v_{j}, v_{i_{1}},...,v_{i_{p}})$ for a pattern instead of the expression given in the definition but it is really important to bear in mind the definition as it sets the first vertex apart in the same way as $v$ is in $(v,P)$ and thus make the combinatorics a bit more straightforward.
\item What we mainly care about is finding the form of the pattern, we don't mind the way the pattern is labelled. This justifies the motivation behind finding isomorphic trees, but as specified in the definition, not all isomorphisms can be convenient as we really want to identify the vertex that connects the pattern to the rest of the tree. Hence the condition over $\varphi$ and condition 2.
\item The last condition is natural as it translates the fact that we are working in the outside part of $T$ or rather that there are no other vertices of $T$ attached to one $v_{i_{k}}$ other than the ones connected in $\widetilde{T}$.
\end{enumerate}
\end{rem}

The proof of our main result relies on the study of a specific random variable on trees which we now define.

\begin{de}
Let $(v, P)$ be a rooted tree with $p + 1$ vertices, and let $n\geq p+1$ be a natural number. For all sets of $p+1$ distinct indices $(j, i_{1}, \cdots, i_{p}) \in \llbracket 1, n\rrbracket^{p+1}$, we define the following function on the set of all trees with $n$ vertices :
\begin{equation*}
\gamma_{j,i_1,...,i_p}(T) =
\begin{cases}
1 &\text{if $(v_{j}, v_{i_{1}}, \cdots, v_{i_{p}})$ is a $(v, P)$-pattern}\\
0 &\text{else}
\end{cases}
\end{equation*}
We also denote by $P_{n}(T)$ the number of $(v, P)$-patterns in $T$, which is given by the formula
\begin{equation*}
P_{n}(T) = \sum_{j=1}^{n}\sum\limits_{\underset{i_{1} \neq j}{i_{1}=1}}^{n}\sum\limits_{\underset{i_{2} \neq j}{i_{2}=1}}^{i_{1}-1}\cdots\sum\limits_{\underset{i_{p} \neq j}{i_{p}=1}}^{i_{p-1}-1}\gamma_{j, i_{1}, \cdots, i_{p}}(T).
\end{equation*}
\end{de}

\begin{rem}
\begin{enumerate}
\item Once the root $v_{j}$ of a pattern is fixed, every permutation of the indices $i_{1}, \cdots, i_{p}$ describes the same pattern alongside $v_{j}$. Thus we need to choose an arbitrary order for the indices $i_{1}, \cdots, i_{p}$. For the sake of clarity, we chose the decreasing order, hence the boundaries of the sums in the expression of $P_{n}$.
\item If we equip the set of all trees with $n$ vertices with the uniform probability measure $\p_n$, then $\gamma$ and $P_{n}$ become random variables. 
% and Theorem \ref{thm:main} can be stated as follows :  \begin{equation*}
% \p_n(P_{n}\geqslant 1) \underset{n\rightarrow +\infty  {\longrightarrow} 1.
% \end{equation*}
% We will prove this result by showing $\p_n(P_{n}=0) \underset{n\to +\infty}{\longrightarrow} 0$ with the help of the Bienaymé-Tchebychev inequality.
\end{enumerate}
\end{rem}

Before turning to the proof of our main result, let us carry out an elementary computation which will be useful later on.

\begin{lem}\label{lem:number_patterns}
Given the $p+1$ indices $j,i_1,...,i_p$, the number of rooted trees (j,T) with vertices $\{j,i_1,...,i_p\}$ isomorphic to $(v,P)$ is  
$$\frac{p!}{\#\mathrm{Aut}(v, P)}$$
where $\mathrm{Aut}(v, P)$ denotes the subgroup of $\mathrm{Aut}(P)$ consisting in graph automorphisms fixing $v$.
\end{lem}

This lemma will explain the apparition of this constant as a factor in all the expressions we will establish later.

\begin{proof}
Let $X$ be the set of all possible rooted trees with vertices labelled by $j, i_{1}, \cdots, i_{p}$ and root $j$. Let the symmetric group $\mathfrak{S}_{p}$ operate on $X$ by setting, for all $\sigma\in \mathfrak{S}_{p}$ and $T = (\{j,i_{1}, \cdots, i_{p}\}, E)\in X$ :
\begin{equation*}
\sigma\cdot T := (\{j, i_{1}, \cdots, i_{p}\}, E')
\end{equation*}
with
\begin{equation*}
E':=\{\{i_{\sigma(k)},i_{\sigma(k')}\}\ |\ \{i_k,i_{k'}\}\in E\}\sqcup\{\{j,i_{\sigma(k)}\}\ |\ \{j, i_{k}\}\in E\}.
\end{equation*}
    
We can see $P$ as an element of $X$ by labelling the root $v$ by $j$ and fixing an arbitrary labelling of the other vertices. Then, its orbit $\omega(P)$ under the action is exactly the set of all rooted trees $(j, T)$ isomorphic to $(v,P)$. Moreover, we have a bijection between $\omega(P)$ and $\sfrac{\mathfrak S_{p}}{\mathrm{Stab}(v, P)}$. Since by definition, we have $\text{Stab}(v,P)=\mathrm{Aut}(v, P)$ it follows that the set of $P$-patterns is in bijection with $\sfrac{\mathfrak{S}_{p}}{\mathrm{Aut}(v, P)}$. The result then follows by considering the cardinals of these sets.
\end{proof}

\begin{comment}
\begin{de}
    Let $P’$ = $(V’, E’)$ be a graph with p+1 vertices. Let $m \in V’$ be a vertex on which we put a mark, who distinguishes the vertex m. The pattern $P = (V_p, E_p)$ marked in m is such that $V_p = V’$ and $E_p = E’ \cup {m}$. (therefore, m is the only singleton in $E_p$). 
\end{de}

\begin{de}
    Let $n \ge p+1$ be a natural number. Let $T_n$ = $(V_n, E_n)$ a tree whose vertices are $v_1,…v_n$. A subgraph $T$ = $(V := {v_j, v_{i1}, …, v_{ip}}, E := E_{n|V})$ of $T_n$ is a P-pattern in $T_n$ marked in $v_j$  if and only if :

    1.	$T_j$ = $(V, E \cup {v_j})$ and P are isomorphic 

    2.	deg$T_n(v_j)$ =deg$T(v_j)$ +1

    3.	For all $w \in V-{v_j}, degT_n(w) = degT(w)$

We write $(v_j, …, v_ip)$ in order to refer to the pattern.
\end{de}
\end{comment}

\section{The main result}\label{sec:main}

We will now state and prove the main result of this work, showing that any fixed pattern appears almost surely in a random tree. Here is a more precise statement.

\begin{thm}\label{thm:main}
Let $P$ be a tree with $p+1$ vertices and let $v\in V(P)$. With the notations of Section \ref{sec:preliminaries}, we have
\begin{equation*}
\p_n(P_{n}\geqslant 1) \underset{n\rightarrow +\infty}{\longrightarrow} 1.
\end{equation*}
\end{thm}

We will prove this result by showing $\p_n(P_{n}=0) \underset{n\to +\infty}{\longrightarrow} 0$. For clarity, the proof will be split into several propositions. From now on, we fix an integer $p$ and a pattern $(v, P)$ with $p+1$ vertices. Following \cite{erdos1963asymmetric} and \cite{junk2020almost}, the strategy is to use the Bienaymé-Tchebychev inequality to estimate the probability of the complementary event, and this first requires computing the expectation and variance of $P_{n}$. Let us start with the expectations of the random variables $\gamma$. Recall that we see them as random variables on the set of all trees with $n$ vertices with the uniform measure.

\begin{prop}\label{prop:expectation_gamma}
For any $1\leqslant j, i_{1}, \cdots, i_{p}\leqslant n$, we have
\begin{equation*}
\E_n(\gamma_{j,i_{1}, \cdots, i_{p}}) =
\begin{cases}
0 & \text{if two indices are equal} \\
\displaystyle\frac{p!}{\#\mathrm{Aut}(v, P)}\frac{(n-(p+1))^{n-(p+2)}}{n^{n-2}} & \text{otherwise}
\end{cases}
\end{equation*}
where $\E_n(X)$ denotes the expected value of $X$ with respect to the probability measure $\p_n$.
\end{prop}

\begin{proof}
The first case follows by definition, hence all we have to do is count the number of possible patterns on the indices. To do this, let us describe how to construct all trees with $n$ vertices in which $(j,i_1,...,i_p)$ is a $(v, P)$-pattern :
\begin{enumerate}
\item According to Cayley's formula, there are $(n-(p+1))^{n-(p+1)-2}$ possible trees we can construct with $(n-(p+1))$ vertices which we can then label with the indices other than $j, i_{1}, \cdots, i_{p}$.
\item By Lemma \ref{lem:number_patterns}, there are $\displaystyle\frac{p!}{\#\mathrm{Aut}(v,P)}$ $(v, P)$-patterns labelled by $j, i_{1}, \cdots, i_{p}$.
\item We can glue each possible pattern to the rest of the tree via one of the $(n-(p+1))$ vertices remaining, giving us $(n-(p+1))$ trees for each of the two choices above.
\end{enumerate}
Conversely, it is clear that any tree with in which $(j,i_1,...,i_p)$ is a $(v, P)$-pattern comes from this construction, hence the result.
\end{proof}

This straightforwardly yields the expectation of $P_{n}$.

\begin{cor}\label{cor:expectation_P}
We have
\begin{equation*}
\E_n(P_n)=n\binom{n-1}{p}\frac{p!}{\#\mathrm{Aut}(v,P)}\frac{(n-(p+1))^{n-(p+2)}}{n^{n-2}}
\end{equation*}
\end{cor}

\begin{proof}
The first statement follows from Proposition \ref{prop:expectation_gamma} by linearity of the expectation. Indeed, we only have to count the number of possible $\gamma_{j, i_{1}, \cdots, i_{p}}$'s :   
\begin{enumerate}
\item We have $n$ choices for $j$.
\item We have $\binom{n-1}{p}$ choices for $\{i_1,...,i_p\}$ once $j$ is fixed.
\end{enumerate}

% \textcolor{red}{Prove the second part (even shortly).}
\end{proof}

The next step is to compute the variances. Once again, we start with $\gamma$.

\begin{lem}\label{lem:variance_gamma}
Let $n\geq 2(p+1)$ be an integer and $j, i_{1}, \cdots, i_{p}, j', i'_{1}, \cdots, i'_{p}$ be indices in $\llbracket 1,n \rrbracket$. Then,
\begin{equation*}
\E_n(\gamma_{j,i_1,...,i_p}\gamma_{j',i'_1,...,i'_p})=
\begin{cases}
\left(\frac{p!}{\#\mathrm{Aut}(v, P)}\right)^{2}\frac{(n-2(p+1))^{n-2(p+1)}}{n^{n-2}} & \text{if all the indices are distinct} \\
\frac{p!}{\#\mathrm{Aut}(v, P)}\frac{(n-(p+1))^{n-(p+2)}}{n^{n-2}} & \text{if $j=j'$ and $\{i_{1}, \cdots, i_{p}\}=\{i'_{1}, \cdots, i'_{p}\}$} \\
0 & otherwise
\end{cases}
\end{equation*}
\end{lem}

\begin{proof}
 We split the proof in three cases.
\begin{enumerate} \item  \emph{First case:} the number of different rooted trees with root $v_{j}$ (resp. $v_{j'}$) over the set $\{j, i_{1}, \cdots, i_{p}\}$ (resp. $\{j', i'_{1}, \cdots, i'_{p}\}$) is $\frac{p!}{\#\mathrm{Aut}(v, P)}$ and we can always construct such trees simultaneously as the sets of indices are distinct. Then, we can construct $(n-2(p+1))^{n-2(p+1)-2}$ different trees over the $n-2(p+1)$ remaining vertices (that are not $j, i_{1}, \cdots, i_{p}, j', i'_{1}, \cdots, i'_{p}$). We have $n-2(p+1)$ places where we can attach each pattern. The conclusion follows.
\item \emph{Second case:} let $A$ be a tree with $n$ vertices such that $j=j'$ and both $(v_{j}, v_{i_{1}}, \cdots, v_{i_{p}})$ and $(v_{j}, v_{i'_{1}}, \cdots, v_{i'_{p}})$ are $(v,P)$-patterns. Let us note $v_j,T$ and $v_j,T'$ the rooted trees of the corresponding pattern. By definition of a pattern we have 
\begin{equation*}
\text{deg}(v_j)=\text{deg}_T(v_j)+1=\text{deg}_{T'}(v_j)+1
\end{equation*}
so that $v_j$ has same degree in both $T$ and $T'$ and has degree $1$ when we remove one of the trees from $A$. As $n\geq 2(p+1)$, $v_j$ is not linked to $T'$ when we suppress $T$ from $A$ (and vice versa). Therefore, any vertex adjacent to $v_j$ in $T$ is also adjacent to $v_{j}$ in $T'$. Let $v$ be adjacent to $v_j$ in $T$ (it exists as $p$ is greater than $1$) then $v_j$ is in $T'$. As $(v_j,T)$ and $(v_j,T')$ are patterns and $v$ is in both $T$ and $T'$ all of its neighborhood in $A$ is in both $T$ and $T'$. Then, the path in $T$ from $v$ to a vertex of $T$ is also fully included in $T'$. Then any vertex of $T$ is a vertex of $T'$. As both trees have the same amount of vertices we conclude that $T=T'$. This proves the \emph{second case} as we are drawn back to Proposition \ref{prop:expectation_gamma}. This also partly proves the last case.
\item \emph{Third case:} let us now finish the proof. We have to show that if $j \neq j'$ but the indices are not distinct then $(v_{j}, v_{i_{1}}, \cdots, v_{i_{p}})$ and $(v_{j}, v_{i'_{1}}, \cdots, v_{i'_{p}})$ can't both be $(v,P)$-patterns. If both are patterns then let $T$ and $T'$ be as in the previous case and let $A$ be the ambient tree. Let $v$ be a vertex that belongs to both $T$ and $T'$. As $T$ is a pattern, the whole neighbourhood of $v$ is in both $T$ and $T'$, hence with the same argument as in the previous case, the path from $v$ to any vertex (different of $v_j$) of $T$ is also fully included in $T'$. As a consequence, apart from $v_j$ and $v_{j'}$, $T$ and $T'$ have the same vertex sets. Now, there exists a path from $v_j$ to $v_{j'}$ consisting in vertices of $T$, except $v_{j'}$. But, $(v_j,T)$ being a pattern, there also exists a path between $v_{j'}$ and $v_j$ consisting in vertices outside $T$ (only keeping the root $j$). Combining both, we constructed a cycle in $A$, which is impossible as $A$ is a tree, hence the result.
\end{enumerate}

This describes all trees with $n$ vertices admitting both $(j, i_{1}, \cdots, i_{p})$ and $(j', i'_{1}, \cdots, i'_{p})$ as patterns, so that the proof is complete.
\end{proof}

Using this we can compute the variance of $P_{n}$.

\begin{prop}\label{prop:variance_P}
Assume that $n\geqslant 2p+2$. Then,
\begin{equation*}
\E_n(P_{n}^{2}) = \left[\frac{n!(n-2(p+1))^{n-2(p+1)}}{(\#\mathrm{Aut}(v, P))^{2}(n-2(p+1))!} + \frac{n!(n-(p+1))^{n-(p+2)}}{\#\mathrm{Aut}(v, P)(n-(p+1))!}\right]\frac 1{n^{n-2}}
\end{equation*}
\end{prop}

\begin{proof}
Writing $P_{n}$ as the sum of all $\gamma$ functions and expanding the square, three types of terms appear, corresponding to the three cases of Lemma \ref{lem:variance_gamma}. The last case yields $0$ and the second one yields the second term in the sum of the statement. Therefore, all we have to do is to count the number of times the first case occurs, that is to say the number of sets of indices $j, i_{1}, \cdots, i_{p}, j', i'_{1}, \cdots, i'_{p'}$ all distinct. This is $0$ if $n\leqslant 2p+2$, hence the assumption. Now we proceed as follows:
\begin{enumerate}
\item We choose the first set of indices: we have $n$ choices for the root and $\binom{n-1}{p}$ choices for the remainder.
\item Then, we choose the second set of indices: we have $(n-(p+1))$ choices for the root and $\binom{n-(p+2)}{p}$ for the remainder.
\end{enumerate}
Thus, we have:
\begin{align*}
\E_n(P_{n}^{2}) & = \frac{1}{n^{n-2}} \left(\frac{n!(n-(p+1))^{n-(p+2)}}{\#\mathrm{Aut}(v, P)(n-(p+1))!}\right. \\
& + \left. n\binom{n-1}{p}(n-(p+1))\binom{n-(p+2)}{p}\frac{p!^2(n-2(p+1))^{n-2(p+1)}}{\#\mathrm{Aut}(v, P)^{2}}\right).
\end{align*}

\vspace{0.4cm}
And by spelling out the binomial coefficients, we observe that :
\begin{align*}
    n\binom{n-1}{p}(n-(p+1))\binom{n-(p+2)}{p} & = \frac{n (n-1)! (n-(p+1))(n-(p+2))!}{p! (n-1-p)! p! (n-(p+2)-p)!} \\
    % & = \frac{n! (n-(p+1))!}{p!^2 (n-(p+1))! (n-2(p+1))!} \\
    & = \frac{n!}{p!^2 (n-2(p+1))!}.
\end{align*}

\vspace{0.4cm}
Finally, substituting this back into our initial formula, we obtain:
\begin{align*}
\E_n(P_{n}^{2}) & = \frac{1}{n^{n-2}} \left( \frac{n!(n-(p+1))^{n-(p+2)}}{\#\mathrm{Aut}(v, P)(n-(p+1))!} + \frac{n!}{p!^2 (n-2(p+1))!}\frac{p!^2(n-2(p+1))^{n-2(p+1)}}{\#\mathrm{Aut}(v, P)^{2}} \right) \\
& = \frac{1}{n^{n-2}} \left(\frac{n!(n-(p+1))^{n-(p+2)}}{\#\mathrm{Aut}(v, P)(n-(p+1))!}  + \frac{n!}{(n-2(p+1))!}\frac{(n-2(p+1))^{n-2(p+1)}}{\#\mathrm{Aut}(v, P)^{2}} \right).
\end{align*}

\end{proof}
\vspace{0.4cm}
We now have everything in hand to prove Theorem \ref{thm:main}.

\begin{proof}[Proof of Theorem \ref{thm:main}]
We will use the complementary event and study $\p_n(P_{n} = 0)$. Observing that if $P_{n} = 0$, then
\begin{equation*}
\vert P_{n} - \E_n(P_{n})\vert = \vert \E_n(P_{n})\vert \geqslant \E_n(P_{n})
\end{equation*}
and using the Bienaymé-Tchebychev inequality, we have
\begin{align*}
\p_n(P_{n} = 0) & \leqslant \p_n(\vert P_{n} - \E_n(P_{n})\vert \geqslant \E_n(P_{n})) \\
& \leqslant \frac{\mathrm{Var}(P_{n})}{\E_n(P_{n})^{2}} \\
& = \frac{\E_n(P_{n}^{2}) - \E_n(P_{n})^{2}}{\E(P_{n})^{2}} \\
& = \frac{\E_n(P_{n}^{2})}{\E_n(P_{n})^{2}} - 1.
\end{align*}

%\textcolor{red}{Put the computations of equivalents here and then conclude ?}
But we have the following asymptotics for $\E_n(P_n)$ and $\E_n(P_n^2)$ :

\begin{align*}
\E_n(P_n) & =n\binom{n-1}{p}\frac{p!}{\#\mathrm{Aut}(v,P)}\frac{(n-(p+1))^{n-(p+2)}}{n^{n-2}} \\
          & =n\frac{(n-1)!}{p!(n-1-p)!}\frac{p!}{\#\mathrm{Aut}(v,P)}\frac{(n-(p+1))^{n-(p+2)}}{n^{n-2}} \\
          & =\frac{n}{\#\mathrm{Aut}(v,P)}(n-1)...(n-p)(n-(p+1))^{-p}\frac{(n-(p+1))^{n-2}}{n^{n-2}} \\
          & =\frac{n}{\#\mathrm{Aut}(v,P)}(1-\frac{p+1}{n})^{n-2}(1+O(\frac 1 n)) \\
          & =\frac{n}{\#\mathrm{Aut}(v,P)}(e^{-(p+1)}+O(\frac 1 n))+O(1)\\
          & =\frac{n}{\#\mathrm{Aut}(v,P)}e^{-(p+1)}+O(1)
\end{align*}

\begin{align*}
    \E_n(P_{n}^{2}) & = \left[\frac{n!(n-2(p+1))^{n-2(p+1)}}{(\#\mathrm{Aut}(v, P))^{2}(n-2(p+1))!} + \frac{n!(n-(p+1))^{n-(p+2)}}{\#\mathrm{Aut}(v, P)(n-(p+1))!}\right]\frac 1{n^{n-2}} \\
    & = \frac{1}{n^{n-2}}\frac{n!(n-2(p+1))^{n-2(p+1)}}{(\#\mathrm{Aut}(v, P))^{2}(n-2(p+1))!}+\frac{n}{\#\mathrm{Aut}(v,P)}e^{-(p+1)}+O(1) \\
    & = \frac{1}{(\#\mathrm{Aut}(v, P))^2}(1-\frac{2(p+1)}{n})^{n-2}\frac{n(n-1)...(n-2p-1)}{(n-2(p+1))^{2p}}  +\frac{n}{\#\mathrm{Aut}(v,P)}e^{-(p+1)}+O(1)\\
    & = \frac{1}{(\#\mathrm{Aut}(v, P))^2}(e^{-2(p+1)}+O(\frac 1 n))(n^2+O(n))+\frac{n}{\#\mathrm{Aut}(v,P)}e^{-(p+1)}+O(1)\\
    & =  \frac{n^2}{(\#\mathrm{Aut}(v, P))^2}e^{-2(p+1)}+O(n)
\end{align*}

Hence
\begin{equation*}
\frac{\E_n(P_n ^2)}{\E_n(P_n)^2} \underset{n \to +\infty}{\sim} 1.
\end{equation*}
Therefore,
\begin{equation*}
\p_n(P_n = 0) \underset{n\to +\infty}{\longrightarrow} 0
\end{equation*}
and eventually, using the complementarity : 
\begin{equation*}
\p_n(P_n \ge 1) \underset{n\to +\infty}{\longrightarrow} 1.
\end{equation*}
\end{proof}

\section{Application to quantum automorphism groups}\label{sec:applications}

Given a finite graph $G$, one can associate to it by \cite{banica2005quantum} a compact quantum group in the sense of \cite{woronowicz1995compact}, called its \emph{quantum automorphism group} and denoted by $\QAut(G)$. It is in general a difficult problem to understand the structure of $\QAut(G)$ from the properties of $G$, or even to know whether it is genuinely quantum (in the sense that the associated C*-algebra is non-commutative). In the case of trees, it was proven in \cite{junk2020almost} that the quantum automorphism group is almost surely infinite, hence genuinely quantum. Our result will now allow us to produce more examples of such asymptotic properties.

\begin{thm}
Let $P$ be a tree. Then, almost surely, the quantum automorphism group of any tree contains $\QAut(P)$ as a quantum subgroup.
\end{thm}

\begin{proof}
Consider a tree $T$ containing the pattern $(v, P)$ for some $v\in V(P)$. It is clear that $\mathrm{Aut}(T)$ contains $\mathrm{Aut}(v, P)$, and the same is true for quantum automorphism groups. It is therefore enough to show that there exists a rooted tree $(v, \widetilde{P})$ such that $\QAut(v, \widetilde{P}) = \QAut(P)$. This can be done using the fact that the center of a finite tree is fixed by all quantum automorphisms (which follows from the fact that quantum automorphisms preserve distances). Therefore, if the center is a vertex $v$, then $\QAut(P) = \QAut(v, P)$ while if the center is an edge, then by adding a vertex $\widetilde{v}$ in the middle, we get a new graph $\widetilde{P}$ such that $\QAut(\widetilde{v}, \widetilde{P}) = \QAut(P)$. We refer the reader to \cite[Sec 5]{dobben2023quantum} for details on that ``rootification'' operation.
\end{proof}

Here is a sample of the properties than one can deduce from this.

\begin{cor}\label{cor:almostsure}
Let $N\geqslant 5$ be an integer. The following properties are true for almost all trees $T$ :
\begin{enumerate}
\item $\QAut(T)$ contains $S_{N}^{+}$;
\item The full C*-algebra $C(\QAut(T))$ is not exact;
\item $\QAut(T)$ is not amenable;
%\item $\QAut(T)$ does not satisfy Kazhdan's property (T).
\end{enumerate}
\end{cor}

\begin{proof}
\begin{enumerate}
\item Let $N\geqslant 5$ and observe that $S_{N}^{+}$ is the quantum automorphism group of the rooted tree with $N$ vertices attached to the root. Therefore, for almost all trees there is a surjective $*$-homomorphism $C(\QAut(T))\to C(S_{N}^{+})$.
\item Since every quotient of an exact C*-algebra is exact (see for instance \cite[Cor 9.4.3]{brown2008finite}) and since $C(S_{N}^{+})$ is not exact by \cite[Thm 3.8]{weber2012classification}, this follows from the first point.
\item If $\QAut(T)$ is amenable, then $C(\QAut(T))$ is nuclear, hence exact. This therefore follows from the second point.
%\item This follows from an argument similar to that of the first point, noticing that $S_{N}^{+}$ does not have property (T) (it has Haagerup's property and is infinite) and that property (T) passes to quantum subgroups by \cite[Prop 6]{fima2008kazhdan}.
\end{enumerate}
\end{proof}

While the authors were completing this work, the preprints \cite{dobben2023quantum} and \cite{meunier2023quantum} were issued. They contain a description of all the quantum automorphism groups of trees, and the vast majority of them satisfy the properties of Corollary \ref{cor:almostsure}. This does however not give an alternate proof of the said corollary. Indeed, even though it is known that $\QAut(T)$ is almost surely non-trivial, there are trees such that their quantum automorphism group is amenable. For instance the tree $T_{4}$ on five vertices with four leaves satisfies $\QAut(T_{4}) = S_{4}^{+}$, and by combining it with large trees having no quantum automorphism, one can produce arbitrarily large trees having amenable quantum automorphism group. It is therefore not clear from their results that such a phenomenon can only occur with asymptotic probability zero; however, our Corollary \ref{cor:almostsure} removes this ambiguity by proving that the phenomenon occurs with probability zero. %That remark is nevertheless not valid for the last point, since one can check directly that $\QAut(T)$ never has Kazhdan's property (T).

\end{document}